\documentclass{amsart}

\newtheorem{theorem}{Theorem}[section]
\newtheorem{lemma}[theorem]{Lemma}
\newtheorem{corollary}[theorem]{Corollary}

\theoremstyle{definition}

\theoremstyle{remark}
\newtheorem{remark}[theorem]{Remark}

\numberwithin{equation}{section}

\usepackage{txfonts}
\begin{document}

\title{Rigidity theorems for  compact Bach-flat manifolds with positive constant scalar curvature}

\author{Haiping Fu}
\address{Department of Mathematics, Nanchang University, Nanchang 330047, P. R. China
}
%\curraddr{Department of Mathematics, Nanchang University, Nanchang 330047, P. R. China}
\email{mathfu@126.com(H. P. Fu)}
\thanks{Supported by the National Natural Science Foundation of China (11261038, 11361041).}

\author{Jianke Peng}
\address{Department of Mathematics, Nanchang University, Nanchang 330047, P. R. China}
\email{pjkdahua@163.com(J. K. Peng)}
\thanks{}

%    General info

\date{May 9, 2016 and, accepted, April 19, 2017.}

\dedicatory{}


\begin{abstract}
In this paper, we prove some rigidity theorems for compact Bach-flat $n$-manifold with the positive constant scalar curvature. In particular, our conditions in Theorem 1.4  have the additional properties of being sharp.
\end{abstract}

\maketitle

\section{Introduction}
\par Let $(M^n,g) (n\geq3)$ be an $n$-dimensional Riemannian manifold with the Riemannian curvature tensor $Rm=\{R_{ijkl}\}$, the Weyl curvature tensor $W=\{W_{ijkl}\}$, the Ricci tensor $Ric=\{R_{ij}\}$ and the scalar curvature $R$. For any manifold of dimension $n\geq4$, the Bach tensor, introduced by Bach \cite{B}, is defined as \begin{eqnarray}
B_{ij}\equiv\frac{1}{n-3}\nabla^k\nabla^lW_{ikjl}+\frac{1}{n-2}R^{kl}W_{ikjl}.
\end{eqnarray}
Here and hereafter the Einstein convention of summing over the repeated indices will be adopted. In \cite{KL}, Korzynski and Lewandowski proved that the Bach tensor can be identified with the Yang-mills current of the  Cartan  normal conformal connection.
Recall that a metric $g$ is called Bach-flat if the Bach tensor vanishes. It is easy to see that $B_{ij}=0$ if $(M^n,g)$ is either locally conformally flat, or an Einstein manifold. In the case of $n=4$, $g$ is Bach-flat if and only if it is a critical metric of the functional (see \cite{{Be},{GN}}) $$W:g\mapsto\int_{M}{|W_g|}^2dV_g.$$

Now we introduce the definition of the Yamabe constant. Given a complete Riemannian $n$-manifold $(M^n,g)$ of dimension $n\geq3$, the Yamabe constant $Y(M,[g])$ ($[g]$ is the conformal class of $g$) is defined as
$$Y(M,[g])\equiv\inf\limits_{\tilde{g}\in[g]}\frac{\int_M{R_{\tilde{g}}}dV_{\tilde{g}}}{\left(\int_MdV_{\tilde{g}}\right)^{\frac{n-2}{n}}}=\inf\limits_{0\neq u\in C^\infty_0(M^n)}\frac{\int_{M}\left(\frac{4(n-1)}{n-2}|\nabla u|^2+Ru^2\right)dV_g}{\left(\int_{M}|u|^\frac{2n}{n-2}dV_g\right)^\frac{n-2}{n}}.$$
The important works of Aubin \cite{A}, Schoen \cite{S}, Trudinger \cite{T} and Yamabe \cite{Y} showed that  for compact manifolds the infimum in the above is always achieved. There are noncompact complete Riemannian manifolds of negative scalar curvature with positive Yamabe constant. For example, any simply connected complete locally conformally flat manifold has positive Yamabe constant (see \cite{SY} ). Furthermore, for compact manifolds, $Y(M,[g])$ is determined by the sign of the scalar curvature $R$ (see \cite{A} ), and for noncompact manifolds, $Y(M,[g])$ is always positive if $R$ vanishes (see \cite{D}).

The curvature pinching phenomenon plays an important role in global differential geometry. Some isolation theorems of Weyl curvature tensor of positive Einstein manifolds are given in \cite{{HV},{IS},{Si}}, when its $L^{\frac{n}{2}}$-norm is small. Recently, two rigidity theorems of Weyl curvature tensor of positive Einstein manifolds are given
in \cite{{C},{FX2},{FX3}}, which improve results due to \cite{{HV},{IS},{Si}}.  The first author and Xiao have studied compact manifolds with harmonic curvature to obtain some rigidity results in \cite{{F},{FX}}. Here when a  Riemannian manifold satisfies $\delta Rm= \{\nabla^lR_{ijkl}\}=0$, we call it a manifold with harmonic curvature.
Bach-flat manifolds have been studied by many authors. For any complete Bach-flat manifold, Kim \cite{K} has studied their rigidity phenomena  and derived that a complete  Bach-flat 4-manifold $M^4$ with nonnegative constant scalar curvature and the positive Yamabe constant is an Einstein manifold if the $L^2$-norm of the trace-free Riemannian curvature tensor $\mathring{Rm}$ is small enough. Later, Chu \cite{Ch} improved Kim's result and showed that $M^4$ is in fact a space of constant curvature under the same assumptions.  Chu and Feng \cite{CF} proved the rigidity result for  $n$-dimensional Bach-flat manifolds with constant scalar curvature and positive Yamabe constant. For a compact Bach-flat manifold $M^4$ with the positive Yamabe constant, Chang et al. \cite{CQY} proved that $M^4$ is conformal equivalent to the standard four-sphere provided that the $L^2$-norm of the Weyl curvature tensor $W$ is small enough, and also showed that there is only finite diffeomorphism class with a bounded $L^2$-norm of $W$.

Now, we are interested in $L^p$ pinching problems for  compact Bach-flat manifolds with positive constant scalar curvature. In this paper, under some $L^p$ pinching conditions,  we show that the  compact Bach-flat manifold with positive constant scalar curvature is spherical space form or Einstein manifold. More precisely, we have the following theorems:
\begin{theorem}
Let $(M^n,g) (n\geq4)$ be an $n$-dimensional  compact Bach-flat Riemannian manifold with positive constant scalar curvature. For $p\geq\frac{n}{2}$, if
$$\left(\int_{M}|\mathring{Rm}|^pdV_g\right)^{\frac{1}{p}}<\varepsilon(n)Y(M,[g])^\frac{n}{2p}R^{1-\frac{n}{2p}},$$
where $\varepsilon(n)$ is a constant depending only on $n$, i.e.,
\[
\varepsilon(n)=\left\{
\begin{array}{cc}
\frac{n-2}{4(n-1)\left(C(n)+(n-2)\sqrt{\frac{n-2}{2(n-1)}}\right)}, &\mbox{if $n=4,5$ and $p=\frac{n}{2}$,}\\
\left[\frac{(n-2)(2p-n)}{n(6-n)}\right]^{\frac{n}{2p}}\frac{(6-n)p}{2(n-1)(2p-n)\left(C(n)+(n-2)\sqrt{\frac{n-2}{2(n-1)}}\right)}, &\mbox{if $n=4,5$ and $\frac{n}{2}<p<\frac{2n}{n-2}$,}\\
\frac{1}{(n-1)\left(C(n)+(n-2)\sqrt{\frac{n-2}{2(n-1)}}\right)}, &\mbox{if $n\geq6$ and $p\geq\frac{n}{2}$ or if $n=4,5$ and $p\geq\frac{2n}{n-2}$,}
\end{array}\right.
\]
and $C(n)$ is defined in Lemma 2.1,
then $(M^n,g)$ is isometric to a quotient of the round $\mathbb{S}^n$.
\end{theorem}
\begin{corollary}
Let $(M^n,g) (n\geq4)$ be an $n$-dimensional  compact Bach-flat Riemannian manifold with positive constant scalar curvature. If
$$\left(\int_{M}|\mathring{Rm}|^{\frac n2}dV_g\right)^{\frac{2}{n}}<\varepsilon(n)Y(M,[g]),$$
where
\[
\varepsilon(n)=\left\{
\begin{array}{cc}
\frac{n-2}{4(n-1)\left(C(n)+(n-2)\sqrt{\frac{n-2}{2(n-1)}}\right)}, &\mbox{if $n=4,5$,}\\
\frac{1}{(n-1)\left(C(n)+(n-2)\sqrt{\frac{n-2}{2(n-1)}}\right)}, &\mbox{if $n\geq6$,}
\end{array}\right.
\]
then $(M^n,g)$ is isometric to a quotient of the round $\mathbb{S}^n$.
\end{corollary}
\begin{remark}
The above $L^p$-pinching condition in Theorem $1.1$ is invariant under any homothety. $L^{\frac n2}$ trace-free Riemannian curvature pinching theorems have been shown by Kim \cite{K}, Chu \cite{C}, and Chu and Feng \cite{CF}, in which the pinching constant  are not explicit, respectively.  Theorem $1.1$ extends the $L^{p}$ trace-free Riemannian curvature pinching theorems given by \cite{{C},{CF},{K}} in power $p=\frac n2$ to  $p\geq\frac n2$.
\end{remark}
\begin{theorem}
Let $(M^n,g) (n\geq4)$ be an $n$-dimensional  compact Bach-flat Riemannian manifold with positive constant scalar curvature. If
\begin{eqnarray}
\left(\int_{M}\left|W+\frac{\sqrt{n}}{2\sqrt{2}(n-2)}\mathring{Ric}\circledwedge g\right|^{\frac{n}{2}}dV_g\right)^{\frac{2}{n}}<C_1(n)Y(M,[g]),
\end{eqnarray}
where
\[
C_1(n)=\left\{
\begin{array}{cc}
\sqrt{\frac{n-2}{32(n-1)}}, &\mbox{if $n=4,5$,}\\
\frac{1}{\sqrt{2(n-2)(n-1)}}, &\mbox{if $n\geq6$,}
\end{array}\right.
\]
then $(M^n,g)$ is an Einstein manifold. In particular, for $n=4,5$, then $(M^n,g)$ is isometric to a quotient of the round $\mathbb{S}^n$; for $n\geq6$, if the pinching constant in (1.2) is weakened to $\frac{2Y(M,[g])}{nC_2(n)}$, where $C_2(n)$ is defined in Lemma 2.1 of \cite{FX3}, then $(M^n,g)$ is isometric to a quotient of the round $\mathbb{S}^n$.
\end{theorem}

\begin{remark}
When $n\geq6$, the inequality (1.2) of this theorem is optimal. The critical case is given by the following example. If $(\mathbb{S}^1(t)\times\mathbb{S}^{n-1},g_t)$ is the product of the circle of radius $t$ with $\mathbb{S}^{n-1}$, and if $g_t$ is the standard product metric normalized such that $Vol(g_t)=1$, we have $W=0$, $g_t$ is a Yamabe metric for small $t$ (see \cite{S2}), and $\left(\int_{M}|\mathring{Ric}|^{\frac{n}{2}}dV_g\right)^{\frac{2}{n}}=\frac{Y(M,[g])}{\sqrt{n(n-1)}}$, which is the critical case of the inequality (1.2) in Theorem $1.4$. We know that $(\mathbb{S}^1(t)\times\mathbb{S}^{n-1},g_t)$ is not Einstein.
\end{remark}
\begin{corollary}
Let $(M^4,g)$ be a $4$-dimensional  compact Bach-flat Riemannian manifold with positive constant scalar curvature. If
 \begin{eqnarray}
 \int_{M}|W|^2dV_g+\frac{5}{4}\int_{M}|\mathring{Ric}|^2dV_g\leq\frac{1}{48}\int_{M}R^2dV_g,
 \end{eqnarray}
then $(M^4,g)$ is isometric to a quotient of the round $\mathbb{S}^4$.
\end{corollary}
\begin{remark}
By the Chern-Gauss-Bonnet formula,
the pinching condition (1.3) in Corollary $1.6$ is equivalent to the following
\begin{eqnarray}
\int_{M}|W|^2dV_g+\frac{2}{39}\int_{M}R^2dV_g\leq\frac{160}{13}\pi^2\chi(M),
\end{eqnarray}
where $\chi(M)$ is the Euler-Poincar\'{e} characteristic of $M$.
\end{remark}
\begin{theorem}
Let $(M^n,g)$ be an $n$-dimensional compact Bach-flat Riemannian manifold with positive constant scalar curvature. If
\begin{eqnarray}
|W|^2+\frac{{n}}{2(n-2)}|\mathring{Ric}|^2\leq\frac{1}{{2(n-2)(n-1)}}R^2,
\end{eqnarray}
then $(M^n,g)$ is isometric to either an Einstein manifold or a quotient of $\mathbb{S}^1\times \mathbb{S}^{n-1}$ with the product metric.
\end{theorem}
\begin{corollary}
Let $(M^n,g)$ be an $n=4$ or $5$-dimensional compact Bach-flat Riemannian manifold with positive constant scalar curvature. If
 $$
 |W|^2+\frac{{n}}{2(n-2)}|\mathring{Ric}|^2\leq\frac{1}{{2(n-2)(n-1)}}R^2,
$$
then $(M^n,g)$ is isometric to either a quotient of the round $\mathbb{S}^n$ or a quotient of $\mathbb{S}^1\times \mathbb{S}^{n-1}$ with the product metric.
\end{corollary}
%% The correct journal style for \specialsection is all uppercase; a known bug
%% in amsart.cls prevents this, so input must be uppercase until it is fixed.
%\specialsection*{This is a Special Section Head}
%%%%%%%%%%%%%%%%%%%%%%%%%%%%%%%%%%%%%%%%%%%%%%%%%%%%%%%%%%%%%%%%%%%%%%%%
%
%%%%%%%%%%%%%%%%%%%%%%%%%%%%%%%%%%%%%%%%%%%%%%%%%%%%%%%%%%%%%%%%%%%%%%%%

\section{Proof of Theorem 1.1}
Let $(M^n,g)(n\geq3)$ be an $n$-dimensional  complete Riemannian manifold with the metric $g=\{g_{ij}\}$. Denote by $Ric=\{R_{ij}\}$ and $R$ the Ricci tensor and the scalar curvature, respectively. It is well known that the Riemannian curvature tensor $Rm=\{R_{ijkl}\}$ of $M^n$ can be decomposed into three orthogonal components which have the same symmetries as
$Rm$ $$R_{ijkl}=W_{ijkl}+V_{ijkl}+U_{ijkl},$$ $$V_{ijkl}=\frac{1}{n-2}(\mathring{R}_{ik}g_{jl}-\mathring{R}_{il}g_{jk}+\mathring{R}_{jl}g_{ik}-\mathring{R}_{jk}g_{il}),$$ $$U_{ijkl}=\frac{R}{n(n-1)}(g_{ik}g_{jl}-g_{il}g_{jk}),$$
where $W=\{W_{ijkl}\}$, $V=\{V_{ijkl}\}$ and $U=\{U_{ijkl}\}$ denote the Weyl curvature tensor, the traceless Ricci part and the scalar curvature part, respectively, and $\mathring{Ric}=\{\mathring{R}_{ij}\}=\{R_{ij}-\frac{R}{n}g_{ij}\}$ is the trace-free Ricci tensor.
Denote by $\mathring{Rm}=\{\mathring{R}_{ijkl}\}=\{R_{ijkl}-U_{ijkl}\}$  the trace-free Riemannian curvature tensor.
 In local coordinates, the norm of a $(0,4)$-type tensor $T$ is defined as $$|T|^2=|T_{ijkl}|^2=g^{im}g^{jn}g^{ks}g^{lt}T_{ijkl}T_{mnst}.$$

The following equalities are easily obtained from the properties of Riemannian curvature tensor:
\begin{equation}
g^{ik}\mathring{R}_{ijkl}=\mathring{R}_{jl,}
\end{equation}
\begin{equation}
\mathring{R}_{ijkl}+\mathring{R}_{iljk}+\mathring{R}_{iklj}=0,
\end{equation}
\begin{equation}
\mathring{R}_{ijkl}=\mathring{R}_{klij}=-\mathring{R}_{jikl}=-\mathring{R}_{ijlk},
\end{equation}
\begin{equation}
|\mathring{Rm}|^2=|W|^2+|V|^2=|W|^2+\frac{4}{n-2}|\mathring{Ric}|^2.
\end{equation}
Moreover, under the assumption of constant scalar curvature, we get
\begin{equation}
\nabla_h\mathring{R}_{ijkl}+\nabla_l\mathring{R}_{ijhk}+\nabla_k\mathring{R}_{ijlh}=0,
\end{equation}
and
\begin{eqnarray}
\nabla^lW_{ijkl}&=&\nabla^lR_{ijkl}-\nabla^lV_{ijkl}-\nabla^lU_{ijkl}\nonumber\\
&=&\nabla^l\mathring{R}_{ijkl}-\nabla^lV_{ijkl}\nonumber\\
&=&\frac{n-3}{n-2}\left(\nabla_j\mathring{R}_{ik}-\nabla_i\mathring{R}_{jk}\right)\nonumber \\
&=&\frac{n-3}{n-2}\nabla^l\mathring{R}_{ijkl}.
\end{eqnarray}
Since $n\geq3$, from (2.4) we see that
\begin{eqnarray}
|\mathring{Ric}|^2\leq\frac{n-2}{4}|\mathring{Rm}|^2.
\end{eqnarray}

Let $\Lambda^2(M)$ and $\otimes^2(M)$  denote the space of skew symmetric $2$-tensors  and  $2$-tensors, respectively. It is easy to know that the dimension of $\Lambda^2(M)$ and $\otimes^2(M)$ is $\frac{n(n-1)}{2}$ and  $n^{2}$, respectively. Let $T=\{T_{ijkl}\}$ be a tensor with the same symmetries as the Riemannian curvature tensor. It defines a symmetric operator $T:\Lambda^2(M)\rightarrow\Lambda^2(M)$ by
$$(T\omega)_{kl}:=T_{ijkl}\omega_{ij},$$
with $\omega\in\Lambda^2(M)$. Similarly, it also defines a symmetric operator $T:\otimes^2(M)\rightarrow\otimes^2(M)$ by
$$(T\theta)_{kl}:=T_{kilj}\theta_{ij},$$
with $\theta\in\otimes^2(M)$.

In order to prove Theorem $1.1$, we need the following lemma:
\begin{lemma}
Let $(M^n,g) (n\geq3)$ be an n-dimensional Riemannian manifold with constant scalar curvature, then
\begin{eqnarray}
\mathring{R}^{ijkl}\Delta\mathring{R}_{ijkl}\geq
-C(n)|\mathring{Rm}|^3+2\mathring{R}^{ijkl}\nabla_l\nabla^m\mathring{R}_{ijkm}+A(n)R|\mathring{Rm}|^2,
\end{eqnarray}
where
\[
A(n)=\left\{
\begin{array}{cc}
\frac{1}{n-1}, &\mbox{if $R\geq0$,}\\ \frac{2}{n}, &\mbox{if $R<0$,}
\end{array}\right.
\]
and $C(n)=\frac{4(n^2-2)}{n\sqrt{n^2-1}}+\frac{n^2-n-4}{\sqrt{(n-2)n(n^2-1)}}+\sqrt{\frac{(n-2)(n-1)}{n}}$.
\end{lemma}
\begin{remark}
Although Lemma $2.1$ and the explicit coefficient of the term $|\mathring{Rm}|^3$ of (2.8) have been proved in \cite{Ch} and \cite{FX} respectively, for completeness, we also write it out.
\end{remark}
\begin{proof} To simplify the notations, we will compute at an arbitrarily chosen point $p\in M$ in normal coordinates centered at $p$ so that $g_{ij}=\delta_{ij}$. We obtain from (2.3) and (2.5) that
\begin{eqnarray}
\mathring{R}_{ijkl}\Delta\mathring{R}_{ijkl}&=&\mathring{R}_{ijkl}\nabla_m\nabla_m\mathring{R}_{ijkl}\nonumber \\ &=&2\mathring{R}_{ijkl}\nabla_m\nabla_l\mathring{R}_{ijkm}\nonumber \\ &=&2\mathring{R}_{ijkl}(\nabla_l\nabla_m\mathring{R}_{ijkm}+R_{hilm}\mathring{R}_{hjkm} +R_{hjlm}\mathring{R}_{ihkm}+R_{hklm}\mathring{R}_{ijhm}+R_{hmlm}\mathring{R}_{ijkh}),
\end{eqnarray}
where the Ricci identities are used in the last equality of (2.9). By the definition of trace-free Riemannian curvature tensor and (2.1), from (2.9) we get
\begin{eqnarray}
\mathring{R}_{ijkl}\Delta\mathring{R}_{ijkl}&=&2\mathring{R}_{ijkl}\nabla_l\nabla_m\mathring{R}_{ijkm}+2\mathring{R}_{ijkl}(R_{hilm}\mathring{R}_{hjkm}+R_{hjlm}\mathring{R}_{ihkm}+R_{hklm}\mathring{R}_{ijhm})+2R_{hl}\mathring{R}_{ijkl}\mathring{R}_{ijkh}\nonumber \\ &=&2\mathring{R}_{ijkl}\nabla_l\nabla_m\mathring{R}_{ijkm}+2\mathring{R}_{ijkl}(\mathring{R}_{hilm}\mathring{R}_{hjkm}+\mathring{R}_{hjlm}\mathring{R}_{ihkm}+\mathring{R}_{hklm}\mathring{R}_{ijhm})\nonumber\\
&&+2\mathring{R}_{ijkl}\mathring{R}_{ijkh}\mathring{R}_{hl}+\frac{2R}{n}|\mathring{Rm}|^2
+\frac{2R}{n(n-1)}\mathring{R}_{ijkl}(\mathring{R}_{ljki}+\mathring{R}_{ilkj}+\mathring{R}_{ijlk}+\mathring{R}_{jk}\delta_{li}-\mathring{R}_{ik}\delta_{lj})\nonumber\\
&=&2\mathring{R}_{ijkl}\nabla_l\nabla_m\mathring{R}_{ijkm}+2\mathring{R}_{ijkl}(\mathring{R}_{hilm}\mathring{R}_{hjkm}
+\mathring{R}_{hjlm}\mathring{R}_{ihkm}+\mathring{R}_{hklm}\mathring{R}_{ijhm})\nonumber\\
& &+2\mathring{R}_{ijkl}\mathring{R}_{ijkh}\mathring{R}_{hl}-\frac{4R}{n(n-1)}|\mathring{Ric}|^2+\frac{2R}{n}|\mathring{Rm}|^2\nonumber\\
&=&2\mathring{R}_{ijkl}\nabla_l\nabla_m\mathring{R}_{ijkm}-2\left(2\mathring{R}_{ijlk}\mathring{R}_{ihlm}\mathring{R}_{hjmk}+\frac{1}{2}\mathring{R}_{ijkl}\mathring{R}_{ijhm}\mathring{R}_{hmkl}\right)\nonumber\\
& &+2\mathring{R}_{ijkl}\mathring{R}_{ijkh}\mathring{R}_{hl}-\frac{4R}{n(n-1)}|\mathring{Ric}|^2+\frac{2R}{n}|\mathring{Rm}|^2.
\end{eqnarray}
\par We consider $\mathring{Rm}$ as a trace-free symmetric operator on $\Lambda^2(M)$ and $\otimes^2(M)$. By the algebraic inequalities $tr(T^3)\leq\frac{m-2}{\sqrt{m(m-1)}}|T|^3$ for trace-free symmetric $m$-matrices $T$ and $\lambda_i\leq\sqrt{\frac{m-1}{m}}|T|$ for the eigenvalues $\lambda_i$ of $T$ in \cite{H}, we get
\begin{eqnarray}
\left|2\mathring{R}_{ijlk}\mathring{R}_{ihlm}\mathring{R}_{hjmk}+\frac{1}{2}\mathring{R}_{ijkl}\mathring{R}_{ijhm}\mathring{R}_{hmkl}\right|
\leq2|\mathring{R}_{ijlk}\mathring{R}_{ihlm}\mathring{R}_{hjmk}|+\frac{1}{2}|\mathring{R}_{ijkl}\mathring{R}_{ijhm}\mathring{R}_{hmkl}|\nonumber\\
\leq\left(\frac{2(n^2-2)}{n{\sqrt{n^2-1}}}+\frac{n^2-n-4}{2\sqrt{(n-2)n(n^2-1)}}\right)|\mathring{Rm}|^3,
\end{eqnarray}
and
\begin{eqnarray}
|\mathring{R}_{ijkl}\mathring{R}_{ijkh}\mathring{R}_{hl}|\leq\sqrt{\frac{n-1}{n}}|\mathring{Ric}||\mathring{Rm}|^2.
\end{eqnarray}
Combining with (2.7), (2.10), (2.11) and (2.12), we get
\begin{eqnarray}
\mathring{R}_{ijkl}\Delta\mathring{R}_{ijkl}&\geq&
-\left(\sqrt{\frac{(n-1)(n-2)}{n}}+\frac{4(n^2-2)}{n\sqrt{n^2-1}}+\frac{n^2-n-4}{\sqrt{(n-2)n(n^2-1)}}\right)|\mathring{Rm}|^3\nonumber\\
& &+2\mathring{R}^{ijkl}\nabla_l\nabla^m\mathring{R}_{ijkm}+A(n)R|\mathring{Rm}|^2.
\end{eqnarray}
\end{proof}
\textbf{Proof of Theorem $1.1$.}
For simplicity of natation, we denote by $(\delta\mathring{Rm})_{ijk}=\nabla^l\mathring{R}_{ijkl}$ the divergence of the trace-free Riemannian curvature tensor and $u=|\mathring{Rm}|$. By the Kato inequality $|\nabla\mathring{Rm}|^2\geq|\nabla|\mathring{Rm}||^2$, we get
\begin{eqnarray}
\mathring{R}^{ijkl}\Delta\mathring{R}_{ijkl}
&\leq&\mathring{R}^{ijkl}\Delta\mathring{R}_{ijkl}+|\nabla\mathring{Rm}|^2-|\nabla|\mathring{Rm}||^2\nonumber\\
&=&\frac{1}{2}\Delta|\mathring{Rm}|^2-|\nabla|\mathring{Rm}||^2\nonumber\\
&=&|\mathring{Rm}|\Delta|\mathring{Rm}|=u\Delta u,
\end{eqnarray}
which together with Lemma $2.1$ and integrating on $M^n$ give
$$\int_{M}u\Delta udV_g\geq-C(n)\int_{M}u^3dV_g+2\int_{M}\mathring{R}^{ijkl}\nabla_l\nabla^m\mathring{R}_{ijkm}dV_g+\frac{R}{n-1}\int_{M}u^2dV_g.$$
Moreover, using the Stokes's theorem, we get
\begin{eqnarray}
\int_{M}|\nabla u|^2dV_g\leq C(n)\int_{M}u^3dV_g+2\int_{M}|\delta\mathring{Rm}|^2dV_g-\frac{R}{n-1}\int_{M}u^2dV_g.
\end{eqnarray}
Since $M^n$ is Bach-flat, we have
$$B_{ij}=\frac{1}{n-3}\nabla^k\nabla^lW_{ikjl}+\frac{1}{n-2}R^{kl}W_{ikjl}=0.$$
Multiplying the above equality by $\mathring{R}^{ij}$ and integrating on $M^n$ give
\begin{eqnarray}
0&=&\frac{1}{n-3}\int_{M}\mathring{R}^{ij}\nabla^k\nabla^lW_{ikjl}dV_g+\frac{1}{n-2}\int_{M}\mathring{R}^{ij}R^{kl}W_{ikjl}dV_g\nonumber\\
&=&-\frac{1}{n-2}\int_{M}\nabla^k\mathring{R}^{ij}\nabla^l\mathring{R}_{ikjl}dV_g+\frac{1}{n-2}\int_{M}\mathring{R}^{ij}\mathring{R}^{kl}W_{ikjl}dV_g\nonumber\\
&=&-\frac{1}{n-2}\int_{M}\frac{1}{2}\left(\nabla^k\mathring{R}^{ij}-\nabla^i\mathring{R}^{kj}\right)\nabla^l\mathring{R}_{ikjl}dV_g+\frac{1}{n-2}\int_{M}\mathring{R}^{ij}\mathring{R}^{kl}W_{ikjl}dV_g\nonumber\\
&=&-\frac{1}{2(n-2)}\int_{M}|\delta\mathring{Rm}|^2dV_g+\frac{1}{n-2}\int_{M}\mathring{R}^{ij}\mathring{R}^{kl}W_{ikjl}dV_g,
\end{eqnarray}
where (2.6) and the second Bianchi identities are used in the second line and the third line of (2.16) respectively. Using (2.7) and the Huisken inequality ( see Lemma 3.4 of  \cite{H})
$$|\mathring{R}^{ik}\mathring{R}^{jl}W_{ijkl}|\leq\sqrt{\frac{n-2}{2(n-1)}}|W||\mathring{Ric}|^2,$$
we have
\begin{eqnarray}
\int_{M}|\delta\mathring{Rm}|^2dV_g\leq\frac{n-2}{2}\sqrt{\frac{n-2}{2(n-1)}}\int_{M}u^3dV_g.
\end{eqnarray}
Combining with (2.15) and (2.17), we obtain
\begin{eqnarray}
\int_{M}|\nabla u|^2dV_g\leq E(n)\int_{M}u^3dV_g-\frac{R}{n-1}\int_{M}u^2dV_g,
\end{eqnarray}
where $E(n)=C(n)+\sqrt{\frac{(n-2)^3}{2(n-1)}}$. From (2.18), using Young's inequality and the H\"{o}lder inequality, we get
\begin{eqnarray}
\int_{M}|\nabla u|^2dV_g
&\leq&\frac{nE(n)}{2p}\epsilon^{-\frac{2p-n}{n}}\int_{M}u^{2+\frac{2p}{n}}dV_g+\left(\frac{(2p-n)\epsilon E(n)}{2p}-\frac{R}{n-1}\right)\int_{M}u^2dV_g\nonumber\\
&\leq&\frac{nE(n)}{2p}\epsilon^{-\frac{2p-n}{n}}\left(\int_{M}u^{\frac{2n}{n-2}}dV_g\right)^{\frac{n-2}{n}}\left(\int_{M}u^pdV_g\right)^{\frac{2}{n}}\\
& &+\left(\frac{(2p-n)\epsilon E(n)}{2p}-\frac{R}{n-1}\right)\int_{M}u^2dV_g\nonumber,
\end{eqnarray}
where $\epsilon$ is a positive constant. By the definition of Yamabe constant $Y(M,[g])$, we get
\begin{eqnarray}
\frac{n-2}{4(n-1)}Y(M,[g])\left(\int_{M}u^{\frac{2n}{n-2}}dV_g\right)^{\frac{n-2}{n}}
&\leq&\int_{M}|\nabla u|^2dV_g+\frac{(n-2)R}{4(n-1)}\int_{M}u^2dV_g\nonumber\\
&\leq&(1+\eta)\int_{M}|\nabla u|^2dV_g+\frac{(n-2)R}{4(n-1)}\int_{M}u^2dV_g,
\end{eqnarray}
where $\eta\geq0$ is a constant. Substituting (2.19) into (2.20), we conclude that
\begin{eqnarray}
\left\{\frac{n-2}{4(n-1)}Y(M,[g])-(1+\eta)\frac{nE(n)}{2p}\epsilon^{-\frac{2p-n}{n}}\left(\int_{M}u^pdV_g\right)^{\frac{2}{n}}\right\}\left(\int_{M}u^{\frac{2n}{n-2}}dV_g\right)^{\frac{n-2}{n}}\nonumber\\
\leq\left\{\frac{(n-2)R}{4(n-1)}+(1+\eta)\left(\frac{(2p-n)\epsilon E(n)}{2p}-\frac{R}{n-1}\right)\right\}\int_{M}u^2dV_g.
\end{eqnarray}

How we select $(\epsilon, \eta)$ to maximize $\left(\int_{M}u^pdV_g\right)^{\frac{2}{n}}$ in (2.21)  is equivalent to a problem of finding $(\epsilon, \eta)$ on a domain $\mathcal{D}$ which minimizes a function
$$F(\epsilon, \eta):=\frac{nE(n)}{2p}(1+\eta)\epsilon^{-\frac{2p-n}{n}}.$$
Here,  the domain $\mathcal{D}$ consists of points $(\epsilon, \eta)$ which satisfies inequalities
$$G(\epsilon, \eta):=\frac{(n-2)R}{4(n-1)}+(1+\eta)\left(\frac{(2p-n)\epsilon E(n)}{2p}-\frac{R}{n-1}\right)\leq0, \ \epsilon>0, \ \eta\geq0.$$
In the case of $p=\frac{n}{2}$, since we have \\
$$
F(\epsilon, \eta)=E(n)(1+\eta)  \ \
\begin{array}{cc}
&\mbox{and}
\end{array}
\ \ \ \
G(\epsilon, \eta)=\frac{R}{n-1}\left(\frac{n-6}{4}-\eta\right),
$$
we can set
$$
\eta=\left\{
\begin{array}{cc}
\frac{n-6}{4}, &\mbox{if $n\geq6$,}\\
0, &\mbox{if $n=4,5$}
\end{array}
\right.
, \ \
\begin{array}{cc}
&\mbox{$\epsilon =$ (any positive number).}
\end{array}
$$
In the case of $p>\frac{n}{2}$. In order to minimize $F(\epsilon, \eta)$ and $G(\epsilon, \eta)=0$, we can set
$$
\eta=\left\{
\begin{array}{cc}
\frac{(n-2)p-2n}{2n}, &\mbox{if $n=4,5$ and $p\geq\frac{2n}{n-2}$ or $n\geq6$}\\
0, &\mbox{if $n=4,5$ and $\frac n2<p<\frac{2n}{n-2}$}
\end{array}
\right.
,$$
$$
\epsilon =\left\{
\begin{array}{cc}
\frac{R}{(n-1)E(n)}, &\mbox{if $n=4,5$ and $p\geq\frac{2n}{n-2}$ or $n\geq6$}\\
\frac{p(6-n)R}{2(n-1)(2p-n)E(n)}, &\mbox{if $n=4,5$ and $\frac n2<p<\frac{2n}{n-2}$.}
\end{array}
\right.
$$
 In conclusion,
we can choose
$$\left(\int_{M}u^pdV_g\right)^{\frac{1}{p}}<\varepsilon(n)Y(M,[g])^\frac{n}{2p}R^{1-\frac{n}{2p}}$$
such that (2.21) implies $\left(\int_{M}u^{\frac{2n}{n-2}}dV_g\right)^{\frac{n-2}{n}}=0$, i.e., $\mathring{Rm}=0$. Hence $(M^n,g)$ is isometric to a quotient of the round $\mathbb{S}^n$.
\section{Proof of Theorem 1.4}
Now, we compute the Laplacian of $|\mathring{Ric}|^2$.
\begin{lemma}
Let $(M^n,g) (n\geq4)$ be a complete Bach-flat $n$-manifold with constant scalar curvature, then
\begin{eqnarray}
\Delta|\mathring{Ric}|^2=2|\nabla\mathring{Ric}|^2-4\mathring{R}_{ij}\mathring{R}_{kl}W_{ikjl}+\frac{2n}{n-2}\mathring{R}_{ij}\mathring{R}_{jk}\mathring{R}_{ik}
+\frac{2R}{n-1}|\mathring{Ric}|^2.
\end{eqnarray}
\end{lemma}
\begin{remark}
Although Lemma $3.1$ has been proved in \cite{CGY}, for completeness, we also write it out.\end{remark}
\begin{proof}We obtain from (2.3) and (2.5) that
\begin{eqnarray}
\Delta|\mathring{Ric}|^2&=&2|\nabla\mathring{Ric}|^2+2\mathring{R}_{ij}\nabla_k\nabla_k\mathring{R}_{ij}\nonumber \\
&=&2|\nabla\mathring{Ric}|^2+2\mathring{R}_{ij}\nabla_k(\nabla_j\mathring{R}_{ik}-\nabla_l\mathring{R}_{ilkj})\nonumber \\
&=&2|\nabla\mathring{Ric}|^2+2\mathring{R}_{ij}\nabla_k\nabla_j\mathring{R}_{ik}+2\mathring{R}_{ij}\nabla_k\nabla_l\mathring{R}_{ikjl}.
\end{eqnarray}
Since the scalar curvature is constant, by the Ricci identities, we get
\begin{eqnarray}
\mathring{R}_{ij}\nabla_k\nabla_j\mathring{R}_{ik}&=&\mathring{R}_{ij}(\nabla_j\nabla_k\mathring{R}_{ik}+\mathring{R}_{hk}R_{hijk}+\mathring{R}_{ih}R_{hkjk})\nonumber\\
&=&\mathring{R}_{ij}\mathring{R}_{hk}R_{hijk}+\mathring{R}_{ij}\mathring{R}_{ih}R_{hj}\nonumber\\
&=&\mathring{R}_{ij}\mathring{R}_{hk}[W_{hijk}+\frac{1}{n-2}(\mathring{R}_{ik}\delta_{hj}+\mathring{R}_{hj}\delta_{ik}-\mathring{R}_{ij}\delta_{hk}-\mathring{R}_{hk}\delta_{ij})\nonumber\\
& & +\frac{R}{n(n-1)}(\delta_{ik}\delta_{jh}-\delta_{ij}\delta_{hk})]+\mathring{R}_{ij}\mathring{R}_{ih}\mathring{R}_{hj}+\frac{R}{n}|\mathring{Ric}|^2\nonumber\\
&=&\mathring{R}_{ij}\mathring{R}_{hk}W_{hijk}+\frac{n}{n-2}\mathring{R}_{ij}\mathring{R}_{jk}\mathring{R}_{ik}+\frac{R}{n-1}|\mathring{Ric}|^2,
\end{eqnarray}
and
\begin{eqnarray*}
\mathring{R}_{ij}\nabla_k\nabla_l\mathring{R}_{ikjl}
&=&\mathring{R}_{ij}\nabla_k\nabla_lW_{ikjl}+\frac{1}{n-2}\mathring{R}_{ij}\nabla_k\nabla_l(\mathring{R}_{ij}\delta_{kl}+\mathring{R}_{kl}\delta_{ij}-\mathring{R}_{il}\delta_{jk}-\mathring{R}_{jk}\delta_{il})
\nonumber\\
&=&\mathring{R}_{ij}\nabla_k\nabla_lW_{ikjl}+\frac{1}{n-2}(\mathring{R}_{ij}\Delta\mathring{R}_{ij}-\mathring{R}_{ij}\nabla_j\nabla_l\mathring{R}_{il}-\mathring{R}_{ij}\nabla_k\nabla_i\mathring{R}_{kj})
\nonumber\\
&=&\mathring{R}_{ij}\nabla_k\nabla_lW_{ikjl}+\frac{1}{n-2}\left[\mathring{R}_{ij}\Delta\mathring{R}_{ij}-\mathring{R}_{ij}\nabla_k(\nabla_k\mathring{R}_{ij}+\nabla_l\mathring{R}_{jlki})\right]\nonumber\\
&=&\mathring{R}_{ij}\nabla_k\nabla_lW_{ikjl}+\frac{1}{n-2}\mathring{R}_{ij}\nabla_k\nabla_l\mathring{R}_{ikjl}.
\end{eqnarray*}
Since $M^n$ is Bach-flat, we get from the above equation that
\begin{eqnarray}
\mathring{R}_{ij}\nabla_k\nabla_l\mathring{R}_{ikjl}=-\mathring{R}_{ij}\mathring{R}_{kl}W_{ikjl}.
\end{eqnarray}
Combining with (3.2), (3.3) and (3.4), we obtain
$$\Delta|\mathring{Ric}|^2=2|\nabla\mathring{Ric}|^2-4\mathring{R}_{ij}\mathring{R}_{kl}W_{ikjl}+\frac{2n}{n-2}\mathring{R}_{ij}\mathring{R}_{jk}\mathring{R}_{ik}
+\frac{2R}{n-1}|\mathring{Ric}|^2.$$
This completes the proof of Lemma $3.1$.
\end{proof}
\begin{lemma}
On every $n$-dimensional Riemannian manifold, the following estimate holds
$$\left|-W_{ijkl}\mathring{R}_{ik}\mathring{R}_{jl}+\frac{n}{2(n-2)}\mathring{R}_{ij}\mathring{R}_{jk}\mathring{R}_{ik}\right|
\leq\sqrt{\frac{n-2}{2(n-1)}}|\mathring{Ric}|^2\left(|W|^2+\frac{n}{2(n-2)}|\mathring{Ric}|^2\right)^{\frac{1}{2}}.$$
\end{lemma}
\begin{remark}
We follow these proofs of Proposition 2.1 in \cite{C} and Lemma 4.7 in \cite{Bo} to prove this lemma. For completeness, we also write it out.
In general, according to the proof of Lemma $3.3$, we can obtain
$$\left|-W_{ijkl}\mathring{R}_{ik}\mathring{R}_{jl}+K\mathring{R}_{ij}\mathring{R}_{jk}\mathring{R}_{ik}\right|
\leq\sqrt{\frac{n-2}{2(n-1)}}|\mathring{Ric}|^2\left(|W|^2+\frac{2(n-2)K^2}{n}|\mathring{Ric}|^2\right)^{\frac{1}{2}},$$
where $K$ is a constant.
\end{remark}
\begin{proof}
First of all we have
$$(\mathring{Ric}\circledwedge g)_{ijkl}=\mathring{R}_{ik}g_{jl}-\mathring{R}_{il}g_{jk}+\mathring{R}_{jl}g_{ik}-\mathring{R}_{jk}g_{il},$$
$$(\mathring{Ric}\circledwedge \mathring{Ric})_{ijkl}=2(\mathring{R}_{ik}\mathring{R}_{jl}-\mathring{R}_{il}\mathring{R}_{jk}),$$
where $\circledwedge$ denotes the Kulkarni-Nomizu product.
An easy computation shows
$$W_{ijkl}\mathring{R}_{ik}\mathring{R}_{jl}=\frac{1}{4}W_{ijkl}(\mathring{Ric}\circledwedge \mathring{Ric})_{ijkl},$$
$$\mathring{R}_{ij}\mathring{R}_{jk}\mathring{R}_{ik}=-\frac{1}{8}(\mathring{Ric}\circledwedge g)_{ijkl}(\mathring{Ric}\circledwedge \mathring{Ric})_{ijkl}.$$
Hence we get the following identity
\begin{eqnarray}
-W_{ijkl}\mathring{R}_{ik}\mathring{R}_{jl}+\frac{n}{2(n-2)}\mathring{R}_{ij}\mathring{R}_{jk}\mathring{R}_{ik}=-\frac{1}{4}\left(W+\frac{n}{4(n-2)}\mathring{Ric}\circledwedge g\right)_{ijkl}(\mathring{Ric}\circledwedge \mathring{Ric})_{ijkl}.
\end{eqnarray}
Since $\mathring{Ric}\circledwedge \mathring{Ric}$ has the same symmetries of the Riemannian curvature tensor, it can be orthogonally decomposed as
$$\mathring{Ric}\circledwedge \mathring{Ric}=T+V'+U',$$
where $T$ is totally trace-free and
$$V'_{ijkl}=-\frac{2}{n-2}\left(\mathring{Ric}^2\circledwedge g\right)_{ijkl}+\frac{2}{n(n-2)}|\mathring{Ric}|^2(g\circledwedge g)_{ijkl},$$
$$U'_{ijkl}=-\frac{1}{n(n-1)}|\mathring{Ric}|^2(g\circledwedge g)_{ijkl},$$
where $\left(\mathring{Ric}^2\right)_{ik}=\mathring{R}_{ip}\mathring{R}_{kp}$. Taking the squared norm one obtains
$$|\mathring{Ric}\circledwedge \mathring{Ric}|^2=8|\mathring{Ric}|^4-8|\mathring{Ric}^2|^2,$$
$$|V'|^2=\frac{16}{n-2}|\mathring{Ric}^2|^2-\frac{16}{n(n-2)}|\mathring{Ric}|^4,$$
$$|U'|^2=\frac{8}{n(n-1)}|\mathring{Ric}|^4.$$
In particular, one has
$$|T|^2+\frac{n}{2}|V'|^2=|\mathring{Ric}\circledwedge \mathring{Ric}|^2+\frac{n-2}{2}|V'|^2-|U'|^2=\frac{8(n-2)}{n-1}|\mathring{Ric}|^4.$$
We now estimate the right hand side of (3.5). Using the fact that $W$ and $T$ are totally trace-free and the Cauchy-Schwarz inequality we obtain
\begin{eqnarray*}
\left|\left(W+\frac{n}{4(n-2)}\mathring{Ric}\circledwedge g\right)_{ijkl}(\mathring{Ric}\circledwedge \mathring{Ric})_{ijkl}\right|^2
&=&\left|\left(W+\frac{n}{4(n-2)}\mathring{Ric}\circledwedge g\right)_{ijkl}(T+V')_{ijkl}\right|^2\\
&=&\left|\left(W+\frac{\sqrt{2n}}{4(n-2)}\mathring{Ric}\circledwedge g\right)_{ijkl}\left(T+\sqrt{\frac{n}{2}}V'\right)_{ijkl}\right|^2\\
&\leq&\left|W+\frac{\sqrt{2n}}{4(n-2)}\mathring{Ric}\circledwedge g\right|^2\left(|T|^2+\frac{n}{2}|V'|^2\right)\\
&=&\frac{8(n-2)}{n-1}|\mathring{Ric}|^4\left(|W|^2+\frac{n}{2(n-2)}|\mathring{Ric}|^2\right).
\end{eqnarray*}
This estimate together with (3.5) concludes this proof.
\end{proof}

\textbf{Proof of Theorem $1.4$.} By Lemmas $3.1$ and $3.3$ and the Kato inequality $|\nabla\mathring{Ric}|^2\geq|\nabla|\mathring{Ric}||^2$, we get
\begin{eqnarray}
|\mathring{Ric}|\Delta|\mathring{Ric}|\geq
-\sqrt{\frac{2(n-2)}{n-1}}|\mathring{Ric}|^2\left(|W|^2+\frac{n}{2(n-2)}|\mathring{Ric}|^2\right)^{\frac{1}{2}}+\frac{R}{n-1}|\mathring{Ric}|^2.
\end{eqnarray}
Set $u=|\mathring{Ric}|$. By (3.6), we compute
\begin{eqnarray}
u^\alpha\Delta u^\alpha&=&u^\alpha\left(\alpha(\alpha-1)u^{\alpha-2}|\nabla u|^2+\alpha u^{\alpha-1}\Delta u\right)\nonumber \\
&=&\frac{\alpha-1}{\alpha}|\nabla u^\alpha|^2+\alpha u^{2\alpha-2}u\Delta u\nonumber \\
&\geq&\frac{\alpha-1}{\alpha}|\nabla u^\alpha|^2-\alpha\sqrt{\frac{2(n-2)}{n-1}}\left(|W|^2+\frac{n}{2(n-2)}u^2\right)^{\frac{1}{2}}u^{2\alpha}+\frac{\alpha R}{n-1}u^{2\alpha}.
\end{eqnarray}
Integrating (3.7) on $M^n$ and using Stoke's theorem, we have
\begin{eqnarray}
0\geq\left(2-\frac{1}{\alpha}\right)\int_{M}|\nabla u^\alpha|^2dV_g-\alpha\sqrt{\frac{2(n-2)}{n-1}}\int_{M}\left(|W|^2+\frac{n}{2(n-2)}u^2\right)^{\frac{1}{2}}u^{2\alpha}dV_g
+\frac{\alpha R}{n-1}\int_{M}u^{2\alpha}dV_g.
\end{eqnarray}
For $2-\frac{1}{\alpha}>0$, by the definition of Yamabe constant and H\"{o}ider inequality, we obtain from (3.8) that
\begin{eqnarray}
0&\geq&\left\{\left(2-\frac{1}{\alpha}\right)\frac{n-2}{4(n-1)}Y(M,[g])
-\alpha\sqrt{\frac{2(n-2)}{n-1}}\left(\int_{M}\left(|W|^2+\frac{n}{2(n-2)}u^2\right)^{\frac{n}{4}}dV_g\right)^{\frac{2}{n}}\right\}\left(\int_{M}u^{\frac{2n\alpha}{n-2}}dV_g\right)^{\frac{n-2}{n}}\nonumber \\
& &+\frac{4\alpha^2-2(n-2)\alpha+n-2}{4\alpha(n-1)}R\int_{M}u^{2\alpha}dV_g.
\end{eqnarray}
\textbf{Case 1.} when $n\geq6$, taking $\alpha=\frac{(n-2)\left(1+\sqrt{1-\frac{4}{n-2}}\right)}{4}$, from (3.9), we get
\begin{eqnarray}
0&\geq&\left\{\frac{Y(M,[g])}{\sqrt{2(n-1)(n-2)}}-\left(\int_{M}\left(|W|^2+\frac{n}{2(n-2)}u^2\right)^{\frac{n}{4}}dV_g\right)^{\frac{2}{n}}\right\}\left(\int_{M}u^{\frac{2n\alpha}{n-2}}dV_g\right)^{\frac{n-2}{n}}.
\end{eqnarray}
Since $W$ is totally trace-free, one has
$$\left|W+\frac{\sqrt{n}}{2\sqrt{2}(n-2)}\mathring{Ric}\circledwedge g\right|^2=|W|^2+\frac{n}{2(n-2)}|\mathring{Ric}|^2$$
and the pinching condition (1.2) implies that $M^n$ is Einstein.\\
\textbf{Case 2.} When $n=4,5$, we have $\frac{4\alpha^2-2(n-2)\alpha+n-2}{4\alpha(n-1)}>0$ for $\alpha>\frac{1}{2}$, and from (3.9), we get
\begin{eqnarray*}
0&\geq&\left\{\left(1-\left(1-\frac{1}{\alpha}\right)^2\right)\sqrt{\frac{n-2}{32(n-1)}}Y(M,[g])-\left(\int_{M}\left(|W|^2+\frac{n}{2(n-2)}u^2\right)^{\frac{n}{4}}dV_g\right)^{\frac{2}{n}}\right\}\left(\int_{M}u^{\frac{2n\alpha}{n-2}}dV_g\right)^{\frac{n-2}{n}}.
\end{eqnarray*}
Taking $\alpha=1$, we have
\begin{eqnarray}
0&\geq&\left\{\sqrt{\frac{n-2}{32(n-1)}}Y(M,[g])-\left(\int_{M}\left(|W|^2+\frac{n}{2(n-2)}u^2\right)^{\frac{n}{4}}dV_g\right)^{\frac{2}{n}}\right\}\left(\int_{M}u^{\frac{2n}{n-2}}dV_g\right)^{\frac{n-2}{n}}.
\end{eqnarray}
Thus the pinching condition (1.2) implies that $M^n$ is Einstein.
\par In particular, for $n=4,5$, the pinching condition (1.2) implies
\begin{eqnarray}
\left(\int_{M}|W|^{\frac{n}{2}}dV_g\right)^{\frac{2}{n}}<\sqrt{\frac{n-2}{32(n-1)}}Y(M,[g])<\frac{2Y(M,[g])}{C_2(n)n},
\end{eqnarray}
where the constant  \[
C_2(n)=\left\{
\begin{array}{cc}
\frac{\sqrt{6}}{2}, &\mbox{if $n=4$,}\\
\frac{8\sqrt{10}}{15}, &\mbox{if $n=5$,}\\
\frac{4(n^2-2)}{n\sqrt{n^2-1}}+\frac{n^2-n-4}{\sqrt{(n-2)(n-1)n(n+1)}},&\mbox{if $n\geq6$}
\end{array}\right.
\]
is defined in Lemma 2.1 of \cite{FX3}. By the rigidity result for positively curved Einstein manifolds (see Theorem 1.1 of \cite{FX3}), (3.12) implies that $M^n$ is isometric to a quotient of the round $\mathbb{S}^n$.
For $n\geq6$, we can choose $\alpha$ such that $\frac{4\alpha^2-2(n-2)\alpha+n-2}{4\alpha(n-1)}>0$ and
$$0\geq\left\{\frac{2Y(M,[g])}{C_2(n)n}-\left(\int_{M}\left(|W|^2+\frac{n}{2(n-2)}u^2\right)^{\frac{n}{4}}dV_g\right)^{\frac{2}{n}}\right\}\left(\int_{M}u^{\frac{2n\alpha}{n-2}}dV_g\right)^{\frac{n-2}{n}}.
$$
From Case 1, the pinching condition (1.2) implies that $M^n$ is Einstein. Hence, the pinching condition (1.2) implies
\begin{eqnarray}
\left(\int_{M}|W|^{\frac{n}{2}}dV_g\right)^{\frac{2}{n}}<\frac{2Y(M,[g])}{C_2(n)n}.
\end{eqnarray}
By the rigidity result for positively curved Einstein manifolds (see Theorem 1.1 of \cite{FX3}), (3.13) implies that $M^n$ is isometric to a quotient of the round $\mathbb{S}^n$.\\
\textbf{Proof of Corollary $1.6$.} To prove Corollary $1.6$, we need the following lemma which was proved by Gursky (see \cite{G}). For completeness, we also write it's proof out.
\begin{lemma}
Let $(M^4,g)$ be a complete 4-dimensional manifold, then the following estimate holds
$$\int_{M}R^2dV_g-12\int_{M}|\mathring{Ric}|^2dV_g\leq Y(M,[g])^2,$$
moreover, the inequality is strict unless $(M^4,g)$ is comformally Einstein.
\end{lemma}
\begin{proof}
By the Chern-Gauss-Bonnet formula (see the Equation 6.31 of \cite{Be})
\begin{eqnarray}
\int_{M}|W|^2dV_g-2\int_{M}|\mathring{Ric}|^2dV_g+\frac{1}{6}\int_{M}R^2dV_g=32\pi^2\chi(M),
\end{eqnarray}
and the conformal invariance of $\int_{M}|W|^2dV_g$, we find that $-2\int_{M}|\mathring{Ric}|^2dV_g+\frac{1}{6}\int_{M}R^2dV_g$ is also conformally invariant.
 Let $\tilde{g}\in[g]$ be a Yamabe metric. Then
\begin{eqnarray*}
Y(M,[g])^2&=&\frac{\left(\int_M{R_{\tilde{g}}}dV_{\tilde{g}}\right)^2}{\int_MdV_{\tilde{g}}}=\int_M{R}_{\tilde{g}}^2dV_{\tilde{g}}\\
&\geq&\int_M{R}_{\tilde{g}}^2dV_{\tilde{g}}-12\int_M|{\mathring{Ric}}_{\tilde{g}}|_{\tilde{g}}^2dV_{\tilde{g}}\\
&=&\int_MR^2dV_g-12\int_M|\mathring{Ric}|^2dV_g.
\end{eqnarray*}
 The equality case follows immediately.
\end{proof}

By Lemma $3.5$, we get
\begin{eqnarray}
\int_{M}|W|^2dV_g+\int_{M}|\mathring{Ric}|^2dV_g-\frac{Y(M,[g])^2}{48}\leq\int_{M}|W|^2dV_g+\frac{5}{4}\int_{M}|\mathring{Ric}|^2dV_g-\frac{1}{48}\int_{M}R^2dV_g.
\end{eqnarray}
Moreover, the inequality is strict unless $(M^4,g)$ is comformally Einstein. In the first case $``\,<"$, Theorem $1.4$ immediately implies  Corollary $1.6$; In the second case $``\,="$, $g$ is conformally Einstein. Since $g$ has constant scalar curvature, $g$ is Einstein from the proof of Obata Theorem (see Proposition 3.1 of \cite{LP}). By the rigidity result for positively curved Einstein manifolds (see Theorem 1.1 of \cite{FX3}), (3.15) implies that $M^4$ is isometric to a quotient of the round $\mathbb{S}^4$.
\section{Proof of Theorem $1.8$}
\noindent
\textbf{Proof of Theorem $1.8$.}  From (3.1), by Lemma $3.3$, we get
\begin{eqnarray}
\Delta|\mathring{Ric}|^2\geq2|\nabla\mathring{Ric}|^2+4\sqrt{\frac{n-2}{2(n-1)}}|\mathring{Ric}|^2\left\{\frac{R}{\sqrt{2(n-1)(n-2)}}-\left(|W|^2+\frac{n}{2(n-2)}|\mathring{Ric}|^2\right)^{\frac{1}{2}}\right\}.
\end{eqnarray}
Note that (1.5) is equivalent that the second of RHS of (4.1) is nonnegative.
By  the maximum principle, from (4.1) we get $\nabla\mathring{Ric}=0$. Since $M^n$ has positive constant scalar curvature, $M^n$ is a manifold with parallel Ricci tensor.
 Hence $M^n$ is a manifold with harmonic curvature. Using the same argument as in the proof of (3.1), we obtain a Weitzenb\"{o}ck formula (see (2.20) in \cite{F})
\begin{eqnarray}
\Delta|\mathring{Ric}|^2=2|\nabla\mathring{Ric}|^2-2\mathring{R}_{ij}\mathring{R}_{kl}W_{ikjl}+\frac{2n}{n-2}\mathring{R}_{ij}\mathring{R}_{jk}\mathring{R}_{ik}
+\frac{2R}{n-1}|\mathring{Ric}|^2.
\end{eqnarray}
By Remark $3.4$, we get
\begin{eqnarray}\left|-W_{ijkl}\mathring{R}_{ik}\mathring{R}_{jl}+\frac{n}{n-2}\mathring{R}_{ij}\mathring{R}_{jk}\mathring{R}_{ik}\right|
\leq\sqrt{\frac{n-2}{2(n-1)}}|\mathring{Ric}|^2\left(|W|^2+\frac{2n}{n-2}|\mathring{Ric}|^2\right)^{\frac{1}{2}}.\end{eqnarray}
Combing (4.2) with (4.3), we have
\begin{eqnarray}
\Delta|\mathring{Ric}|^2\geq2|\nabla\mathring{Ric}|^2+2|\mathring{Ric}|^2\left\{\frac{1}{n-1}R-\sqrt{\frac{n-2}{2(n-1)}}\left(|W|^2+\frac{2n}{n-2}|\mathring{Ric}|^2\right)^{\frac{1}{2}}\right\}.
\end{eqnarray}
\textbf{Case 1.} $W\neq 0$. By (4.4),
the inequality (1.5) implies
\begin{eqnarray}
0&\geq&2\sqrt{\frac{n-2}{2(n-1)}}|\mathring{Ric}|^2\left\{\sqrt{\frac{2}{(n-1)(n-2)}}R-\left(|W|^2+\frac{2n}{n-2}|\mathring{Ric}|^2\right)^{\frac{1}{2}}\right\}\nonumber\\
&\geq&
2\sqrt{\frac{n-2}{2(n-1)}}|\mathring{Ric}|^2\left\{\left(4|W|^2+\frac{2n}{n-2}|\mathring{Ric}|^2\right)^{\frac{1}{2}}-\left(|W|^2+\frac{2n}{n-2}|\mathring{Ric}|^2\right)^{\frac{1}{2}}\right\}.
\end{eqnarray}
By (4.5), we get $\mathring{Ric}=0$, i.e., $(M^n, {g})$ is  Einstein.\\
\textbf{Case 2.} $W=0$. Note that the inequality $n(n-1)|\mathring{Ric}|^2\leq R^2$ is equivalent to equation (1.5).  From (4.4),  we have
\begin{eqnarray}
0=\frac{2n}{n-2}\mathring{R}_{ij}\mathring{R}_{jk}\mathring{R}_{ik}
+\frac{2R}{n-1}|\mathring{Ric}|^2\geq\frac{2}{n-1}|\mathring{Ric}|^2\left({R}-\sqrt{(n-1)n}|\mathring{Ric}|\right)\geq0.
\end{eqnarray}
Hence
at every point, either
$\mathring{Ric}$ is null, i.e., $M^n$ is Eninstein, and by conformally flatness it has constant positive sectional curvature, or ${R}-\sqrt{(n-1)n}|\mathring{Ric}|=0$, according to the estimation of trace-free symmetric $2$-tensors, it has an eigenvalue of multiplicity
$(n-1)$ and another of multiplicity $1$. Since the Ricci tensor is parallel, by the de Rham decomposition Theorem, $M^n$ is covered isometrically by the product of Einstein manifolds. We have $R=\sqrt{(n-1)n}|\mathring{Ric}|.$ Since $M^n$ is conformally flat and has positive
scalar curvature, then the only possibility is that $M^n$ is covered isometrically by $\mathbb{S}^1\times \mathbb{S}^{n-1}$ with the product metric.
 So $(M^n,g)$ is isometric to either an Einstein manifold or a quotient of $\mathbb{S}^1\times \mathbb{S}^{n-1}$ with the product metric.\\
\textbf{Proof of Corollary $1.9$.} By Theorem $1.8$, we consider the case that $M^n$ is Einstein.  Using the same argument as in the proof of (3.1), we obtain a Weitzenb\"{o}ck formula for Einstein manifolds (see (5) in \cite{FX3})\begin{eqnarray}
\Delta|W|^2=2|\nabla W|^2-2C_3(n)|W|^3+\frac{4R}{n}|W|^2,
\end{eqnarray}
where $C_3(4)=\frac{\sqrt{6}}{2}$ and $C_3(5)=\frac{8\sqrt{10}}{15}$.
 From (4.7), the condition of Corollary $1.9$ implies that
$$
\Delta|W|^2=2|\nabla W|^2+\left(\frac{4R}{n}-2C_3(n)|W|\right)|W|^2\geq0.
$$
Hence $W=0$, i.e., $M^n$ is conformally flat. So $(M^n,g)$ is isometric to  a quotient of the round $\mathbb{S}^n$. This completes the proof of Corollary $1.9$.
\begin{remark}
Let $(M^n,g) (n\geq4)$ be an $n$-dimensional  compact Bach-flat Riemannian manifold with positive constant scalar curvature. If
$$
|W|^2+\frac{{n}}{2(n-2)}|\mathring{Ric}|^2<\frac{1}{{2(n-2)(n-1)}}R^2,
$$
then $M^n$ is an Einstein manifold.
\end{remark}
\textbf{Acknowledgments:} The authors thank the referee for his helpful suggestions.

\end{document}